\documentclass[12pt]{article}
\usepackage{amssymb}
\usepackage{amsmath,amsthm}
\usepackage[latin1]{inputenc}
\usepackage{hyperref}
\usepackage{enumerate}
\usepackage{tikz}
\usepackage{graphics}
\usepackage{graphicx}
\usepackage{amsthm,amsfonts, amsbsy}
\usepackage[english]{babel}
\usepackage{color}
\usepackage{tkz-berge}
\usepackage{float}
\usepackage[T1]{fontenc}
\usepackage{amssymb,tikz,float,enumerate}

\usetikzlibrary{arrows.meta,shapes.arrows}

\hypersetup{colorlinks=true}

\hypersetup{colorlinks=true, linkcolor=blue, citecolor=blue,urlcolor=blue}

 \newtheorem{remark}{Remark}

 \newtheorem{theorem}[remark]{Theorem}
 
 \newtheorem{corollary}[remark]{Corollary}

\title{Dominating the direct product of two graphs through total Roman strategies\thanks{The third author was partially supported by Slovenian research agency under the grants P1-0297, J1-1693 and J1-9109.}}

\author{Abel Cabrera Mart\'inez$^{(1)}$, Dorota Kuziak$^{(2)}$, Iztok Peterin$^{(3,4)}$ and Ismael G. Yero$^{(5)}$\\
\\
$^{(1)}$ {\small Departament d'Enginyeria Inform\`atica i Matem\`atiques, Universitat Rovira i Virgili}\\
{\small\it abel.cabrera\@@urv.cat}\\
$^{(2)}${\small Departamento de Estad\'istica e Investigaci\'on Operativa, Universidad de C\'adiz}\\
{\small\it dorota.kuziak\@@uca.es}\\
$^{(3)}${\small Faculty of Electrical Engineering and Computer Science, University of Maribor}\\
$^{(4)}${\small Institute of Mathematics, Physics and Mechanics, Ljubljana, Slovenia}\\
{\small\it iztok.peterin\@@um.si} \\
$^{(5)}${\small Departamento de Matem\'aticas, Universidad de C\'adiz}\\
{\small\it ismael.gonzalez\@@uca.es}
}

\date{}

\begin{document}

\maketitle

\begin{abstract}
Given a graph $G$ without isolated vertices, a total Roman dominating function for $G$ is a function $f : V(G)\rightarrow \{0,1,2\}$ such that every vertex with label 0 is adjacent to a vertex with label 2, and the set of vertices with positive labels induces a graph of minimum degree at least one. The total Roman domination number $\gamma_{tR}(G)$ of $G$ is the smallest possible value of $\sum_{v\in V(G)}f(v)$ among all total Roman dominating functions $f$. The total Roman domination number of the direct product $G\times H$ of the graphs $G$ and $H$ is studied in this work. Specifically, several relationships, in the shape of upper and lower bounds, between $\gamma_{tR}(G\times H)$ and some classical domination parameters for the factors are given. Characterizations of the direct product graphs $G\times H$ achieving small values ($\le 7$) for $\gamma_{tR}(G\times H)$ are presented, and exact values for $\gamma_{tR}(G\times H)$ are deduced, while considering various specific direct product classes.
\end{abstract}



\section{Introduction}

The present investigation is devoted to describe a number of contributions to the theory of total Roman dominating functions while dealing with the direct (or tensor or Kronecker) product of two graphs. Studies concerning parameters in relation to domination in graphs are very frequently present in the last recent years. This might probably be caused by the popularity of some classical problems, like for instance Vizing's conjecture \cite{vizing-1,vizing-2} for domination in Cartesian products\footnote{The conjecture claims that the cardinality of the smallest dominating set of the Cartesian product of two graphs is at least equal to the product of the domination numbers of the factor graphs involved in the product.}. See \cite{survey-vizing}, for a survey and recent results concerning this conjecture. Several other problems concerning domination parameters in product graphs have occupied the mind of a significant number of investigators. Works of that type concerning direct product graphs are \cite{bresar2007,Defant2018,Mekis2010,Shiu2012}.

The (total) Roman domination variants are among the most popular topics of domination in graphs. Both versions have had their birth in connection with some defense strategies related to the ancient Roman Empire (see \cite{Revelle2000,Stewart1999}). Studies on (total) Roman domination in product graphs have not escaped from the researchers attention. For instance, \cite{Cabrera,Campanelli,Sumenjak,Yero} are aimed to these goals, although no works appear that considers the Roman domination parameters for the case of direct products. We hence continue with giving new contributions to the theory of parameters related to domination in graph products, specifically we center our attention on the total Roman domination version for the case of the direct product of graphs.

In this work, we consider simple graphs without vertices of degree 0. For a map $f:V(G)\rightarrow \{0,1,2\}$ and a set of vertices $S\subseteq V(G)$, the \emph{weight} of $S$ under $f$ is $f(S)=\sum_{v\in S}f(v)$. Moreover, the \emph{weight} of $f$ is $\omega(f)=f(V(G))$. Since the function $f$ generates three sets $V_0,V_1,V_2$ such that $V_i=\{v\in V(G)\,:\,f(v)=i\}$, $i\in\{0,1,2\}$, we shall write $f=(V_0,V_1,V_2)$.

A function $f=(V_0,V_1,V_2)$ is known to be a \emph{Roman dominating function} on $G$ whenever all vertices $v\in V_0$ have at least one neighbor $u\in V_2$. In connection with this, the parameter of $G$ called \emph{Roman domination number} stands for the least weight among all functions that are proved to be Roman dominating on $G$. This parameter is usually represented as $\gamma_R(G)$. Such concepts in the theory of graphs were formally introduced in \cite{Cockayne2004}, motivated in part by some domination strategies which arose from the antique Roman Empire (see for instance \cite{Revelle2000,Stewart1999}). A Roman dominating function $f=(V_0,V_1,V_2)$ is called a \emph{total Roman dominating function} if  $V_1\cup V_2$ induces a graph without vertices of degree 0. The \emph{total Roman domination number} of $G$ stands for the minimum possible weight among all total Roman dominating functions on $G$. This parameter is pointed out as $\gamma_{tR}(G)$. By a $\gamma_{tR}(G)$-\emph{function} we mean a total Roman dominating function whose weight equals precisely $\gamma_{tR}(G)$. These concepts of total Roman domination were first introduced in \cite{Liu2013} by using some more general settings. The concepts were further specifically introduced and firstly well studied in \cite{Abdollahzadeh2016}.

A set $D=\{v_1,\dots,v_r\}\subset V(G)$ is called a \emph{packing set} of $G$, if $N[v_i]\cap N[v_j]=\emptyset$ for every two different integers $i,j\in \{1,\dots,r\}$. The \emph{packing number} of $G$ is the cardinality of a largest possible packing set of $G$. We represent such cardinality as $\rho(G)$. A packing set induces a subgraph of maximum degree 0, \emph{i.e.}, a graph without edges.
If we substitute the closed neighborhoods with open neighborhood in the definition above, then the concept of \emph{open packing sets} arises. Hence, $D$ is considered to be an \emph{open packing set} whenever $N(v_i)\cap N(v_j)=\emptyset$ for any two distinct $i,j\in \{1,\dots,r\}$. Similarly, the parameter called \emph{open packing number} of $G$ is the cardinality of the largest possible open packing set of $G$. We write this cardinality by using the notation $\rho_o(G)$. We recall that any open packing set represents a set of vertices of the graph which induces a graph having the maximum degree equal to one, and clearly, it could have some vertices whose degree equals zero.

%A set $D\subset V(G)$ is a \emph{dominating set} if every vertex from $V(G)-D$ is adajacent to a vertex from $D$. The minimum cardinality of a dominating set $D$ of $G$ is called the \emph{domination number} of $G$ and is denoted by $\gamma(G)$. A dominating set of cardinality $\gamma(G)$ is called a $\gamma(G)$-set. If there exists a dominating set of $G$ that is at the same time also a closed packing, then $G$ is \emph{efficient closed dominated graph}.

A set $D\subseteq V(G)$ is \emph{total dominating} if all the vertices of the whole graph $G$ have at least a neighbor in the set $D$. The cardinality of the smallest total dominating set of $G$ is known as the \emph{total domination number} of $G$. This cardinality is then represented as $\gamma_t(G)$. A set being total dominating and having cardinality $\gamma_t(G)$ is said to be a $\gamma_t(G)$-\emph{set}. The graph $G$ is called an \emph{efficient open domination graph}, if there is a total dominating set of $G$ which is simultaneously also an open packing.

The \emph{direct product} (also known as tensor product or Kronecker product) of two graphs $G$ and $H$ is the graph denoted by $G\times H$ whose vertex set is given by $V(G\times H)=V(G)\times V(H)$ and the edge set is the Cartesian product of the vertex sets of the factors. That is, $E(G\times H)=\{(g,h)(g',h'):gg'\in E(G), hh'\in E(H)\}$. In Figure~\ref{figure} we show the graph $P_6\times P_6$. As usual, we call the map $p_G:(g,h)\mapsto g$ a \emph{projection} of $G\times H$ onto $G$ and the map $p_H:(g,h)\mapsto h$ a projection of $G\times H$ onto $H$. The set $G^h=\{(g,h):g\in V(G)\}$ is called a $G$-\emph{layer} through $h\in V(H)$ and contains all vertices that project to $h$. An $H$-\emph{layer} $H^g=\{(g,h):h\in V(H)\}$ through $g\in V(G)$ is similarly defined. Note that vertices from a $G$-layer and from an $H^g$-layer form independent sets of $G\times H$.

The direct product is a graph product (see the exhausting monograph on graph products \cite{ImKl}) in categorical sense, as the end vertices of every edge from $G\times H$ project to end vertices of edges in both factors. Consequently, one of the most natural products among all graph products is precisely the direct product, but on the other hand, this also makes this product the most elusive one in many perspectives. So, the connectedness of both factors $G$ and $H$ does not imply the connectedness of the product $G\times H$. (Notice that $P_6\times P_6$ from Figure \ref{figure} is not connected.) To achieve this, one of the factors must also be non-bipartite, see Theorem 5.9 in \cite{ImKl}. One reason for this is that layers form independent sets in $G\times H$. On the other side, the open neighborhoods behave ``nice'', with respect to the factors, while making a direct product based on the fact
\begin{equation}
N_{G\times H}(g,h)=N_G(g)\times N_H(h). \label{neighb}
\end{equation}

\noindent Two different total Roman dominating functions on $P_6\times P_6$ are presented on Figure \ref{figure}.

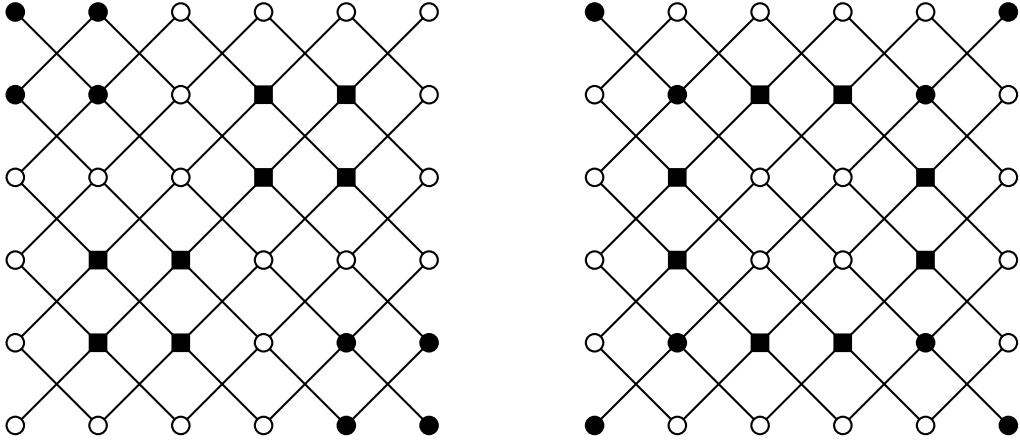
\begin{figure}[ht]
\begin{center}
\begin{tikzpicture}[scale=1.1,style=thick,x=1cm,y=1cm]
\def\vr{3pt} % \vr = vertex radius;

% define vertices
%%%%%
	%%% left
%%%%%
\path (0,0) coordinate (v11);
\path (1,0) coordinate (v21);
\path (2,0) coordinate (v31);
\path (3,0) coordinate (v41);
\path (4,0) coordinate (v51);
\path (5,0) coordinate (v61);

\path (0,1) coordinate (v12);
\path (1,1) coordinate (v22);
\path (2,1) coordinate (v32);
\path (3,1) coordinate (v42);
\path (4,1) coordinate (v52);
\path (5,1) coordinate (v62);

\path (0,2) coordinate (v13);
\path (1,2) coordinate (v23);
\path (2,2) coordinate (v33);
\path (3,2) coordinate (v43);
\path (4,2) coordinate (v53);
\path (5,2) coordinate (v63);

\path (0,3) coordinate (v14);
\path (1,3) coordinate (v24);
\path (2,3) coordinate (v34);
\path (3,3) coordinate (v44);
\path (4,3) coordinate (v54);
\path (5,3) coordinate (v64);

\path (0,4) coordinate (v15);
\path (1,4) coordinate (v25);
\path (2,4) coordinate (v35);
\path (3,4) coordinate (v45);
\path (4,4) coordinate (v55);
\path (5,4) coordinate (v65);

\path (0,5) coordinate (v16);
\path (1,5) coordinate (v26);
\path (2,5) coordinate (v36);
\path (3,5) coordinate (v46);
\path (4,5) coordinate (v56);
\path (5,5) coordinate (v66);

%  edges
	\draw (v11)--(v66);
	\draw (v21)--(v65);
	\draw (v31)--(v64);
	\draw (v41)--(v63);
	\draw (v51)--(v62);
	\draw (v12)--(v21);
	\draw (v13)--(v31);
	\draw (v14)--(v41);
	\draw (v15)--(v51);
	\draw (v16)--(v61);
	
	\draw (v12)--(v56);
	\draw (v13)--(v46);
	\draw (v14)--(v36);
	\draw (v15)--(v26);
	\draw (v26)--(v62);
	\draw (v36)--(v63);
	\draw (v46)--(v64);
	\draw (v56)--(v65);

			\draw (v11)[fill=white] circle (\vr);
			\draw (v21)[fill=white] circle (\vr);
			\draw (v31)[fill=white] circle (\vr);
			\draw (v41)[fill=white] circle (\vr);
			\draw (v51)[fill=black] circle (\vr);
			\draw (v61)[fill=black] circle (\vr);
			\draw (v12)[fill=white] circle (\vr);
			\draw (v22)[fill=black] (1.1,0.9) rectangle (0.9,1.1);
			\draw (v32)[fill=black] (2.1,0.9) rectangle (1.9,1.1);
			\draw (v42)[fill=white] circle (\vr);
			\draw (v52)[fill=black] circle (\vr);
			\draw (v62)[fill=black] circle (\vr);
			\draw (v13)[fill=white] circle (\vr);
			\draw (v23)[fill=black] (1.1,1.9) rectangle (0.9,2.1);
			\draw (v33)[fill=black] (2.1,1.9) rectangle (1.9,2.1);
			\draw (v43)[fill=white] circle (\vr);
			\draw (v53)[fill=white] circle (\vr);
			\draw (v63)[fill=white] circle (\vr);
			\draw (v14)[fill=white] circle (\vr);
			\draw (v24)[fill=white] circle (\vr);
			\draw (v34)[fill=white] circle (\vr);
			\draw (v44)[fill=black] (3.1,2.9) rectangle (2.9,3.1);
			\draw (v54)[fill=black] (4.1,2.9) rectangle (3.9,3.1);
			\draw (v64)[fill=white] circle (\vr);
			\draw (v15)[fill=black] circle (\vr);
			\draw (v25)[fill=black] circle (\vr);
			\draw (v35)[fill=white] circle (\vr);
			\draw (v45)[fill=black] (3.1,3.9) rectangle (2.9,4.1);
			\draw (v55)[fill=black] (4.1,3.9) rectangle (3.9,4.1);
			\draw (v65)[fill=white] circle (\vr);
			\draw (v16)[fill=black] circle (\vr);
			\draw (v26)[fill=black] circle (\vr);
			\draw (v36)[fill=white] circle (\vr);
			\draw (v46)[fill=white] circle (\vr);
			\draw (v56)[fill=white] circle (\vr);
			\draw (v66)[fill=white] circle (\vr);

%%%%%%%%%%%%%right

\path (7,0) coordinate (u11);
\path (8,0) coordinate (u21);
\path (9,0) coordinate (u31);
\path (10,0) coordinate (u41);
\path (11,0) coordinate (u51);
\path (12,0) coordinate (u61);

\path (7,1) coordinate (u12);
\path (8,1) coordinate (u22);
\path (9,1) coordinate (u32);
\path (10,1) coordinate (u42);
\path (11,1) coordinate (u52);
\path (12,1) coordinate (u62);

\path (7,2) coordinate (u13);
\path (8,2) coordinate (u23);
\path (9,2) coordinate (u33);
\path (10,2) coordinate (u43);
\path (11,2) coordinate (u53);
\path (12,2) coordinate (u63);

\path (7,3) coordinate (u14);
\path (8,3) coordinate (u24);
\path (9,3) coordinate (u34);
\path (10,3) coordinate (u44);
\path (11,3) coordinate (u54);
\path (12,3) coordinate (u64);

\path (7,4) coordinate (u15);
\path (8,4) coordinate (u25);
\path (9,4) coordinate (u35);
\path (10,4) coordinate (u45);
\path (11,4) coordinate (u55);
\path (12,4) coordinate (u65);

\path (7,5) coordinate (u16);
\path (8,5) coordinate (u26);
\path (9,5) coordinate (u36);
\path (10,5) coordinate (u46);
\path (11,5) coordinate (u56);
\path (12,5) coordinate (u66);

%  edges
	\draw (u11)--(u66);
	\draw (u21)--(u65);
	\draw (u31)--(u64);
	\draw (u41)--(u63);
	\draw (u51)--(u62);
	\draw (u12)--(u21);
	\draw (u13)--(u31);
	\draw (u14)--(u41);
	\draw (u15)--(u51);
	\draw (u16)--(u61);
	
	\draw (u56)--(u65);
	\draw (u46)--(u64);
	\draw (u36)--(u63);
	\draw (u26)--(u62);
	\draw (u15)--(u26);
	\draw (u12)--(u56);
	\draw (u13)--(u46);
	\draw (u14)--(u36);
	
			\draw (u11)[fill=black] circle (\vr);
			\draw (u21)[fill=white] circle (\vr);
			\draw (u31)[fill=white] circle (\vr);
			\draw (u41)[fill=white] circle (\vr);
			\draw (u51)[fill=white] circle (\vr);
			\draw (u61)[fill=black] circle (\vr);
			\draw (u12)[fill=white] circle (\vr);
			\draw (u22)[fill=black] circle (\vr);
			\draw (u32)[fill=black] (9.1,0.9) rectangle (8.9,1.1);
			\draw (u42)[fill=black] (10.1,0.9) rectangle (9.9,1.1);
			\draw (u52)[fill=black] circle (\vr);
			\draw (u62)[fill=white] circle (\vr);
			\draw (u13)[fill=white] circle (\vr);
			\draw (u23)[fill=black] (8.1,1.9) rectangle (7.9,2.1);
			\draw (u33)[fill=white] circle (\vr);
			\draw (u43)[fill=white] circle (\vr);
			\draw (u53)[fill=black] (11.1,1.9) rectangle (10.9,2.1);
			\draw (u63)[fill=white] circle (\vr);
			\draw (u14)[fill=white] circle (\vr);
			\draw (u24)[fill=black] (8.1,2.9) rectangle (7.9,3.1);
			\draw (u34)[fill=white] circle (\vr);
			\draw (u44)[fill=white] circle (\vr);
			\draw (u54)[fill=black] (11.1,2.9) rectangle (10.9,3.1);
			\draw (u64)[fill=white] circle (\vr);
			\draw (u15)[fill=white] circle (\vr);
			\draw (u25)[fill=black] circle (\vr);
			\draw (u35)[fill=black] (9.1,3.9) rectangle (8.9,4.1);
			\draw (u45)[fill=black] (10.1,3.9) rectangle (9.9,4.1);
			\draw (u55)[fill=black] circle (\vr);
			\draw (u65)[fill=white] circle (\vr);
			\draw (u16)[fill=black] circle (\vr);
			\draw (u26)[fill=white] circle (\vr);
			\draw (u36)[fill=white] circle (\vr);
			\draw (u46)[fill=white] circle (\vr);
			\draw (u56)[fill=white] circle (\vr);
			\draw (u66)[fill=black] circle (\vr);

\end{tikzpicture}
\end{center}
\caption{Two total Roman dominating functions on $P_6\times P_6$ where vertices in $V_0$ are white circles, vertices in $V_1$ are black circles and black squares represent vertices in $V_2$.}\label{figure}
\end{figure}

The \emph{degree} $\delta_G(v)$ of the vertex $v$ in $G$ is represented as the cardinality of the open neighborhood of $v$, that is $\delta_G(v)=|N_G(v)|$. The \emph{maximum degree} of a vertex in a graph $G$ is denoted by $\Delta(G)$. Clearly, $1\leq \Delta(G)\leq |V(G)|-1$ as we consider only simple graphs having no vertices of degree zero. A \emph{leaf} of $G$ is a vertex $v\in V(G)$ with degree $\delta_G(v)=1$ and in contrast, if $\delta_G(v)=|V(G)|-1$, then the vertex $v$ is called as \emph{universal vertex}. For the specific case of the direct product of two graphs $G$ and $H$, we recall that $\delta_{G\times H}(g,h)=\delta_G(g)\delta_H(h)$ and $\Delta(G\times H)=\Delta(G)\Delta(H)$ by (\ref{neighb}).

\section{General bounds}\label{section-general}

We start our exposition with some lower and upper bounds for $\gamma_{tR}(G\times H)$ which are mainly depending on $\rho(G)$, $\rho(H)$, $\gamma_{tR}(G)$ and $\gamma_{tR}(H)$.

\begin{theorem}
\label{th:general bounds}
If $g=(A_0,A_1,A_2)$ is a $\gamma_{tR}(G)$-function (with maximum cardinality of $A_2$) and $h=(B_0,B_1,B_2)$ is a $\gamma_{tR}(H)$-function (with maximum cardinality of $B_2$), then
$$\max\{\rho(H)\gamma_{tR}(G),\rho(G)\gamma_{tR}(H)\}\le \gamma_{tR}(G\times H)\le \gamma_{tR}(H)\gamma_{tR}(G)-2|A_2||B_2|.$$
\end{theorem}

\begin{proof}
We consider a function $f$ on $G\times H$ defined as follows. If $(u,v)\in (A_2\times (B_1\cup B_2))\cup (A_1\times B_2)$, then $f(u,v)=2$; if $(u,v)\in (A_1\times B_1)$, then $f(u,v)=1$; and $f(u,v)=0$ otherwise. If $f(u,v)\ge 1$, then since $g(u)\ge 1$ and $h(v)\ge 1$, there exist two vertices $u'\in N_G(u)$ and $v'\in N_H(v)$ such that $g(u')\ge 1$ and $h(v')\ge 1$. Thus, it follows $(u',v')\in N_{G\times H}(u,v)$ and $f(u',v')\ge 1$. Now, consider a vertex $(u,v)\in V(G\times H)$ such that $f(u,v)=0$. If $(u,v)\in A_0\times V(H)$, then there exist two vertices $u''\in N_G(u)$ and $v''\in N_H(v)$ such that $g(u'')=2$ and $h(v'')\geq 1$. Thus, it follows $(u'',v'')\in N_{G\times H}(u,v)$ and $f(u'',v'')=2$. Finally, if $(u,v)\in A_i\times B_0$ with $i\in \{1,2\}$, then a symmetrical argument to the above one produce a similar conclusion.

As a consequence, we deduce $f$ is a total Roman dominating function on the direct product $G\times H$, which leads to
\begin{align*}
  \gamma_{tR}(G\times H) & \le \omega(f) \\
                         & = 2|A_2\times B_2|+2|A_2\times B_1|+2|A_1\times B_2|+|A_1\times B_1| \\
                         & =(2|A_2|+|A_1|)(|B_2|+|B_1|)+|A_1||B_2| \\
                         & =(2|A_2|+|A_1|)(2|B_2|+|B_1|)-2|A_2||B_2|\\
                         & =\gamma_{tR}(G)\gamma_{tR}(H)-2|A_2||B_2|.
\end{align*}

Now, in order we deduce the lower bound, a $\gamma_{tR}(G\times H)$-function $f$ and a $\rho(G)$-set $S=\{u_1,\dots,u_{\rho(G)}\}$ are considered. Hence, for any integer $i\in \{1,\dots,\rho(G)\}$, we construct a function $h_i$ on $H$ as follows. Also, for any vertex $v\in V(H)$, $h_i(v)=\max\{f(u,v)\,:\,u\in N_G[u_i]\}$.

If $h_i(v)\ge 1$, then there is a vertex $(u,v)\in N_G[u_i]\times \{v\}$ for which $f(u,v)\ge 1$. If $f(u_i,v)=0$, then there exists a vertex $(x,y)\in N_G(u_i)\times N_H(v)$ such that $f(x,y)=2$ and $(x,y)\in N_{G\times H}(u_i,v)$. Moreover, note that in this case $h_i(y)=2$ and also that $y\in N_H(v)$. Now, if $f(u_i,v)\ge 1$, then there exists a vertex $(x',y')\in N_{G\times H}(u_i,v)$ such that $f(x',y')\ge 1$. In such situation, we similarly get $h_i(y')\ge 1$ and  $y'\in N_H(v)$.

On the other hand, if $h_i(v)=0$, then for every vertex $(u,v)\in N_G[u_i]\times \{v\}$ we have $f(u,v)=0$. Particularly, for the vertex $(u_i,v)$, there exists a vertex $(u_i',v')\in N_{G\times H}(u_i,v)$ with $v'\ne v$ and $f(u_i',v')=2$. Hence, for the vertex $v'\in V(H)$ it is satisfied $v'\in N_H(v)$ and $h_i(v')=2$.

As a consequence of these arguments, we deduce that $h_i$ is a total Roman dominating function on $H$ whose weight is less than or equal to  $f(N_G[u_i]\times V(H))$, \emph{i.e.}, $\gamma_{tR}(H)\le f(N_G[u_i]\times V(H))$. Hence, we have the following.
$$\gamma_{tR}(G\times H) \ge \sum_{i=1}^{\rho(G)}f(N_G[u_i]\times V(H))
                          \ge \sum_{i=1}^{\rho(G)} \gamma_{tR}(H)
                          = \rho(G)\gamma_{tR}(H).$$
By the symmetry of the product, we also deduce that $\gamma_{tR}(G\times H)\ge \rho(H)\gamma_{tR}(G)$, and this ends the proof for the case of the lower bound.
\end{proof}

Since every graph of order at least three contains at least one total Roman dominating function whose weight equals the total Roman domination number and at least one vertex labeled two, the following result is directly deduced from the result above.

\begin{corollary}\label{cor1}
For every graphs $G$ and $H$ without vertices of degree 0 and of orders at least three,
$$\gamma_{tR}(G\times H)\le \gamma_{tR}(G)\gamma_{tR}(H)-2.$$
\end{corollary}

Notice that we can avoid the remarks about maximum cardinality of $A_2$ and $B_2$ in Theorem \ref{th:general bounds}. However, the bound is better if we take a $\gamma_{tR}(G)$-function and a $\gamma_{tR}(H)$-function with maximum cardinality of $A_2$ and $B_2$, respectively. The proof of the upper bound from Theorem \ref{th:general bounds} remains valid for any total Roman dominating functions $g$ and $h$ of graphs $G$ and $H$ without isolated vertices, respectively, as long as we exchange $\gamma_{tR}(G)$ and $\gamma_{tR}(H)$ by $\omega(g)$ and $\omega(h)$, respectively, in the last step of the proof. Therefore, we can improve the upper bound of Theorem \ref{th:general bounds}, we we next show.

\begin{remark}\label{remark}
For every two graphs $G$ and $H$ without vertices of degree 0,
$$\gamma_{tR}(G\times H)\le \min\{\omega(g)\omega(h)-2|A_2||B_2|\},$$
where such minimum value is understood for every total Roman dominating functions $g=(A_0,A_1,A_2)$ and $h=(B_0,B_1,B_2)$ on $G$ and $H$, respectively.
\end{remark}

Despite the fact that the bound above represents an advance with respect to the upper bound of Theorem \ref{th:general bounds}, we have no knowledge of one pair of graphs $G$ and $H$ where the bound given in Remark \ref{remark} is better than the upper bound of Theorem \ref{th:general bounds}.

Let $D_G$ be a $\gamma_t(G)$-set. Clearly, the function $g=(V(G)-D_G,\emptyset,D_G)$ total Roman dominating for $G$ and the weight of $g$ is $\omega(g)=2\gamma_t(G)$. Remark \ref{remark} yields the following connection.

\begin{corollary}\label{total}
For any graphs $G$ and $H$ without vertices of degree 0,
$$\gamma_{tR}(G\times H)\le 2\gamma_{t}(G)\gamma_{t}(H).$$
\end{corollary}

If the graphs $G$ and $H$ represents efficient open domination graphs, then $\rho_o(H)=\gamma_t(H)$ and $\rho_o(G)=\gamma_t(G)$ (see Observation 1.1 from \cite{KuPeYe1}), and  Corollary \ref{total} implies the following.

\begin{corollary}
\label{th:two-EOD}
If the graphs $G$ and $H$ represents efficient open domination graphs, then $\gamma_{tR}(G\times H)\le 2\rho_o(G)\rho_o(H)$.
\end{corollary}

A graph $G$ is known to be a \emph{total Roman graph} if it satisfies that $\gamma_{tR}(G)=2\gamma_t(G)$. In the case of two total Roman graphs we can develop the upper bound of Corollary \ref{total} to the following result.

\begin{corollary}\label{cor2}
If $G$ and $H$ are two total Roman graphs, then
$$\gamma_{tR}(G\times H)\le \frac{\gamma_{tR}(G)\gamma_{tR}(H)}{2}.$$
\end{corollary}

The bound given in Theorem \ref{th:general bounds} can be enhanced by a factor of $2$, whenever one factor is bipartite and the other without triangles as shown next.

\begin{theorem}
\label{th:triangle-free-bipar}
If $G$ is a triangle free graph and $H$ is a bipartite graph of order at least two without isolated vertices, then
$$\gamma_{tR}(G\times H)\ge 2\rho(G)\gamma_{tR}(H).$$
\end{theorem}

\begin{proof}
Let $f$ and $S$ be defined in a similar manner to that of the proof of Theorem \ref{th:general bounds} for the lower bound. Clearly, for any vertex $u_i\in S$, $N_G[u_i]\times V(H)$ induces a non connected graph with at least two components. In this sense, for every $i\in \{1,\dots, \rho(G)\}$ and for every component of the subgraph induced by $N_G[u_i]\times V(H)$, we can construct a total Roman dominating function in the same style as in the proof of Theorem \ref{th:general bounds}. This means that $f(N_G[u_i]\times V(H))\ge 2\gamma_{tR}(H)$. A similar argument as the one used to prove Theorem \ref{th:general bounds} gives the stated bound.
\end{proof}

By using a similar argument as the one used while proving the lower bound of Theorem \ref{th:general bounds}, but using an open packing instead of a packing, we can also deduce the lower bound in the following result.

\begin{theorem}\label{EOD}
For any graphs $G$ and $H$ without vertices of degree 0 and of orders at least three,
$$\gamma_{tR}(G\times H)\ge \max\left\{\frac{\rho_o(H)\gamma_{tR}(G)}{2},\frac{\rho_o(G)\gamma_{tR}(H)}{2}\right\}.$$
\end{theorem}

\begin{proof}
The stated bound is obtained by considering a similar partition of the vertex of $G$, as the one used while proving the lower bound of Theorem \ref{th:general bounds}, but instead of using only one vertex as the ``center'' of each set of the partition, we might need to use now two adjacent vertices as the ``centers''. This is based on the structure of open packing sets.

We consider a $\gamma_{tR}(G\times H)$-function $f$, and a $\rho_o(G)$-set $S=S_0\cup S_1$ such that $S_0$ induces a graph without edges and $S_1$ induces a regular graph of degree 1. Note that $S_0$ or $S_1$ could be empty (although not both at the same time). Now, for every $u_i\in S_0$, we construct a function $h_i(v)=\max\{f(u,v)\,:\,u\in N_G[u_i]\}$ for every $v\in V(H)$.

In the same manner, as in the proof of the lower bound of Theorem \ref{th:general bounds}, we deduce that $h_i$ is a total Roman dominating function on $H$, and so, $\gamma_{tR}(H)\le f(N_G[u_i]\times V(H))$ for every $u_i\in S_0$.

Now, for any pair of adjacent vertices $w_i,w'_i\in S_1$, we construct a function $h'_i$ on $H$ as follows. For every $v\in V(H)$, $h'_i(v)=\max\{f(w,v)\,:\,w\in N_G(w_i)\cup N_G(w'_i)\}$. From now on, let $N_i=N_G(w_i)\cup N_G(w'_i)$ and note that $w_i,w'_i\in N_i$.

If $h'_i(v)\ge 1$, then there exists a vertex $(w,v)\in N_i\times \{v\}$ for which $f(w,v)\ge 1$. Assume for instance, that the vertex $w$ is a neighbor of $w_i$ in $G$ (note that $w$ could be $w'_i$). If $f(w_i,v)=0$, then there exists a vertex $(x,y)\in N_G(w_i)\times N_H(v)$ such that $f(x,y)=2$ and $(x,y)\in N_{G\times H}(w_i,v)$. Also, $h'_i(y)=2$ and $y\in N_H(v)$. If $f(w_i,v)\ge 1$, then there exists a vertex $(x',y')\in N_{G\times H}(w_i,v)$ such that $f(x',y')\ge 1$. In such situation, we get $h'_i(y')\ge 1$ and  $y'\in N_H(v)$ as well.

Now, if $h'_i(v)=0$, then for every vertex $(w,v)\in N_i\times \{v\}$ we have $f(w,v)=0$. Particularly, for the vertex $(w_i,v)$ (or for $(w'_i,v)$ as well), there exists a vertex $(z,v')\in N_{G\times H}(w_i,v)$ with $v'\ne v$ and $f(z,v')=2$. Thus, for the vertex $v'\in V(H)$ we have $v'\in N_H(v)$ and $h'_i(v')=2$.

As a consequence of these arguments, we deduce that $h'_i$ is a total Roman dominating function on $H$ whose weight is less than or equal to  $f(N_i\times V(H))=f((N_G(w_i)\cup N_G(w'_i))\times V(H))$, \emph{i.e.}, $\gamma_{tR}(H)\le f((N_G(w_i)\cup N_G(w'_i))\times V(H))$ for every pair of adjacent vertices $w_i,w'_i\in S_1$. Hence, we have the following.
\begin{align*}
  \gamma_{tR}(G\times H) & \ge\sum_{u_i\in S_0}f(N_G[u_i]\times V(H))+\sum_{w_i,w'_i\in S_1,w_i\sim w'_i}f((N_G(w_i)\cup N_G(w'_i))\times V(H)) \\
   & \ge \left(|S_0|+\frac{|S_1|}{2}\right)\gamma_{tR}(H) \\
   & \ge \frac{\rho_o(G)\gamma_{tR}(H)}{2}.
\end{align*}

By the symmetry of the product, we also deduce that $\gamma_{tR}(G\times H)\ge \frac{\rho_o(H)\gamma_{tR}(G)}{2}$, which completes the first part of the proof.
\end{proof}

The bound of Theorem \ref{EOD} can be improved if we consider one bipartite factor and the other without triangles as next stated.

\begin{theorem}
\label{th:triangle-free-regular}
If $G$ is a graph having no triangles and having a $\rho_o(G)$-set which induces a graph with all components isomorphic to $K_2$, and $H$ is a bipartite graph  without vertices of degree 0 and of order at least two, then
$$\gamma_{tR}(G\times H)\ge \rho_o(G)\gamma_{tR}(H).$$
\end{theorem}

\begin{proof}
Let $f$ be a $\gamma_{tR}(G\times H)$-function, and assume $S=\{u_1,v_1,\dots,u_{\rho_o(G)/2},v_{\rho_o(G)/2}\}$ is a $\rho_o(G)$-set such that $u_i\sim v_i$ for every $i\in \{1,\dots,\rho_o(G)/2\}$. Since $H$ is bipartite and $G$ is triangle free, the set $(N(u_i)\cup N(v_i))\times V(H)$ induces a non connected graph with at least two components. In concordance with this fact, by using similar arguments as those ones in the proofs for the lower bounds of Theorems \ref{th:general bounds} and \ref{EOD}, we deduce that for every $i\in \{1,\dots,\rho_o(G)/2\}$, we can construct two total Roman dominating functions $h_i,h'_i$ on $H$ satisfying that $2\gamma_{tR}(H)\le \omega(h_i)+\omega(h'_i)\le f((N_G(u_i)\cup N_G(v_i))\times V(H))$. Therefore, we obtain that
$$\gamma_{tR}(G\times H)\ge \sum_{i=1}^{\rho_o(G)/2}f((N_G(u_i)\cup N_G(v_i))\times V(H))=\frac{\rho_o(G)}{2}(\omega(h_i)+\omega(h'_i))\ge \rho_o(G)\gamma_{tR}(H),$$
and the proof is completed.
\end{proof}

%%%%%%%%%%%%%%%%%%%%%%%%%%%%%%%%%%%%%%%%%%%%%%%%%%%%%%%%%%%%%%%%%

\section{Direct product graphs with small $\gamma_{tR}(G\times H)$}

We concentrate our attention in this section on the case when $\gamma_{tR}(G\times H)$ is small. We shall characterize all the direct products graphs $G\times H$ for which $\gamma_{tR}(G\times H)\leq 7$. For this we need the following class of graphs.

A graph $G$ is called \emph{triangle centered} if there exists a triangle $C_3=xyz$ in $G$ such that every vertex of $G$ is adjacent to at least two vertices of $C_3$. We call such $C_3$ as the \emph{central triangle} of a triangle centered graph. Notice that any two vertices of a central triangle form a total dominating set of a triangle centered graph $G$ and we have $\gamma_t(G)=2$.

\begin{theorem}\label{small}
The following assertions holds for any two graphs $G$ and $H$ without vertices of degree 0.
\begin{itemize}
\item[(i)] There are no graphs $G$ and $H$ for which $\gamma_{tR}(G\times H)\in\{1,2,3,5\}$ .
\item[(ii)] $\gamma_{tR}(G\times H)=4$ if and only if $G$ and $H$ are both isomorphic to $K_2$.
\item[(iii)] $\gamma_{tR}(G\times H)=6$ if and only if ($G$ and $H$ have at least two universal vertices each and at least one of them is of order at least three), or (one factor is $K_2$ and the other one is of order at least three and contains a universal vertex), or (the graphs $G$ and $H$ are triangle centered).
\item[(iv)] $\gamma_{tR}(G\times H)=7$ if and only if both $G$ and $H$ have a universal vertex, one of the graph $G$ and $H$ has exactly one universal vertex, and the other one is different from $K_2$, and only one of $G$ and $H$ can be triangle centered.
\item[(v)] If at most one of the graphs $G$ and $H$ has a universal vertex, $\gamma_t(G)=\gamma_t(H)=2$, and $G$ and $H$ are not both triangle centered, then $\gamma_{tR}(G\times H)=8$.
\end{itemize}
\end{theorem}

\begin{proof}

For $(i)$ notice that there must be at least two adjacent vertices $(g,h)$ and $(g',h')$ in $V_1\cup V_2$ for a $\gamma_{tR}(G\times H)$-function $f=(V_0,V_1,V_2)$. If $|V_1\cup V_2|=2$, then $(g,h')$ and $(g',h)$ have label $0$ and no neighbor with label $2$, a contradiction. This already shows that $\gamma_{tR}(G\times H)\geq 3$. If $\gamma_{tR}(G\times H)=3$, then either $|V_1\cup V_2|=2$, which is not possible, or $|V_1\cup V_2|=3$. In later case there are three vertices of label $1$ and no vertex of label $2$, a contradiction as we have $|V(G\times H)|\geq 4$. Hence $\gamma_{tR}(G\times H)>3$.

To end with $(i)$ suppose that $\gamma_{tR}(G\times H)=5$. Let first $|V_2|=2$ where $(g,h),(g_1,h_1)\in V_2$. If $g\neq g_1$ and $h\neq h_1$, then only one vertex from $(g,h_1)$ and $(g_1,h)$ can have label $1$ and the other has label $0$ and is not adjacent to a vertex of label $2$, a contradiction. So, either $g=g_1$ or $h=h_1$, say $g=g_1$. In $V_1$ is only one vertex, say $(g_2,h_2)$, and it must be adjacent to both vertices of $V_2$. This means that $h_2\neq h$ and $h_2\neq h_1$. But then $(g,h_2)$ posses label $0$ and is not adjacent to a vertex of label $2$, a contradiction.

So let $|V_2|=1$ where $(g,h)\in V_2$ and $(g',h')\in V_1$ is adjacent to $(g,h)$. There are only two more vertices in $V_1$ and these vertices must be $(g,h')$ and $(g',h)$ because they are not adjacent to $(g,h)$. If there exists any other vertex from the mentioned four, then such a vertex implies the existence of a vertex of label $0$ in $G^h\cup H^g$, a contradiction. Hence we have only four vertices and $G\times H\cong K_2\times K_2$. But in this case we have $\gamma_{tR}(G\times H)\leq 4$ as there exists a total Roman dominating function with $V_1=V(G)\times V(H)$. This is the final contradiction and $\gamma_{tR}(G\times H)\neq 5$.

The implication $(\Leftarrow)$ of item $(ii)$ follows from $(i)$ and the total Roman dominating function with $V_1=V(K_2)\times V(K_2)$. For $(\Rightarrow)$ of $(ii)$ suppose that at least one of $G$ and $H$ contains more than three vertices. Hence $|V(G)\times V(H)|\geq 6$ and if all vertices have label $1$, then $\gamma_{tR}(G\times H)\geq 6>4$. Otherwise, if $V_0\neq \emptyset$, then also $V_2\neq \emptyset$. Let $(g,h)\in V_2$ and let $(g',h')\in V_1\cup V_2$ be a neighbor of $(g,h)$. If also $(g,h'),(g',h)\in V_1\cup V_2$, then we have $\gamma_{tR}(G\times H)>4$. On the other hand, if at least one of $(g,h')$ and $(g',h)$ has label $0$, then there exists a vertex of label $2$ different than $(g,h)$ and $(g',h')$, meaning that $\gamma_{tR}(G\times H)>4$ again and $(ii)$ is done.

For $(iii)$ we start with $(\Leftarrow)$. We know from $(i)$ and $(ii)$ that $\gamma_{tR}(G\times H)\geq 6$ whenever at least one of $G$ and $H$ contains more than two vertices, which is true in all three cases. Suppose first that each $G$ and $H$ have at least two universal vertices $g,g'$ and $h,h'$, respectively, and are of order at least three. If we set $V_2=\{(g,h),(g',h')\}$, $V_1=\{(g,h'),(g',h)\}$ and $V_0=V(G)-(V_1\cup V_2)$, then $f_1=(V_0,V_1,V_2)$ is a total Roman dominating function with $\omega(f)=6$. Assume now that one factor, say $H$, is $K_2$ and that $G$ contains at least three vertices together with a universal vertex $g$. For $V(H)=\{h,h'\}$ we define $f_2=(V'_0,V'_1,V'_2)$ by making $V'_2=\{(g,h),(g,h')\}$, $V'_1=\{(g',h'),(g',h)\}$ and $V'_0=V(G)-(V_1\cup V_2)$ for an arbitrary neighbor $g'$ of $g$ in $G$. It is easy to check that $f_2$ is a total Roman dominating function with $\omega(f_2)=6$. The third possibility is that both $G$ and $H$ are triangle centered graphs with central triangles $g_1g_2g_3$ and $h_1h_2h_3$, respectively. We define $V''_2=\{(g_1,h_1),(g_2,h_2),(g_3,h_3)\}$, $V''_1=\emptyset$ and $V''_0=V(G)-V_2$. We will show that $f_3=(V''_0,V''_1,V''_2)$ is a total Roman dominating function. First notice that $V_2$ induces a triangle in $G\times H$. Let $(g,h)\in V_0$. By the definition of the central triangle, $g$ and $h$ are adjacent to at least two vertices of $\{g_1,g_2,g_3\}$ and $\{h_1,h_2,h_3\}$, respectively. Hence, there exists $i\in \{1,2,3\}$ such that $gg_i\in E(G)$ and $hh_i\in E(H)$, and $(g,h)$ is adjacent to $(g_i,h_i)\in V_2$. Therefore, $f$ is a total Roman dominating function on $G\times H$ with $\omega(f_3)=6$. In all three cases we have $\gamma_{tR}(G\times H)\leq 6$ and by $(i)$ and $(ii)$ the equality $\gamma_{tR}(G\times H)=6$ follows.

For the opposite implication $(\Rightarrow)$ of $(iii)$ we have $\gamma_{tR}(G\times H)=6$ and analyze the different possibilities for the cardinalities of $V_1$ and $V_2$ for a $\gamma_{tR}(G\times H)$-function $f=(V_0,V_1,V_2)$. We start with $|V_1|=0$ and $|V_2|=3$ and let $(g_1,h_1),(g_2,h_2),(g_3,h_3)\in V_2$. As $V_1\cup V_2$ induces a graph without isolated vertices one vertex of these mentioned three, say $(g_2,h_2)$, must be adjacent to the other two. Hence $g_1g_2,g_2g_3\in E(G)$ and $h_1h_2,h_2h_3\in E(H)$. If $g_1g_3\notin E(G)$, then $(g_1,h_2)$ is a vertex of label $0$ not adjacent to a vertex from $V_2$. Similar, if $h_1h_3\notin E(H)$, then $(g_2,h_1)$ is a vertex of label $0$ not adjacent to a vertex from $V_2$. Hence $g_1g_2g_3$ and $h_1h_2h_3$ form a triangle in $G$ and $H$, respectively. Suppose that there exists $g\in V(G)$ that is either adjacent to exactly one vertex of $\{g_1,g_2,g_3\}$, say to $g_1$, or to no vertex of $\{g_1,g_2,g_3\}$. In both cases we obtain $(g,h_1)$ must has label $0$, and is not adjacent to any vertex of $V_1\cup V_2$, which is not possible for a total Roman dominating function $f$. Thus, every vertex $g\in V(G)$ must be adjacent to at least two vertices from $\{g_1,g_2,g_3\}$ and $G$ is triangle centered. Similarly one shows that $H$ is triangle centered and the third option follows.

We continue with $|V_1|=2$ and $|V_2|=2$. Let $(g,h)$ and $(g',h')$ be vertices of label 2. Assume first that $(g,h)$ and $(g',h')$ are adjacent. Hence, the vertices $(g,h')$ and $(g',h)$ are not adjacent to $(g,h)$ nor to $(g',h')$ and must have label 1. All the other vertices are in $V_0$. Moreover, $V_0\neq \emptyset$ as the converse leads to a contradiction with $f$ being a $\gamma_{tR}(G\times H)$-function. Every vertex $(g,x)$, $x\in V(H)-\{h,h'\}$ has label $0$ and is not adjacent to $(g,h)$. Therefore they must be adjacent to $(g',h')$, which means that $h'$ is a universal vertex of $H$. Similarly, every vertex $(g',x)$, $x\in V(H)-\{h,h'\}$ has label $0$ and is not adjacent to $(g',h')$. So they are adjacent to $(g,h)$, and $h$ is a universal vertex of $H$. By symmetric arguments, also $g$ and $g'$ are universal vertices of $G$. Thus, both $G$ and $H$ have at least two universal vertices. If both have only two vertices, then we have a contradiction with $(ii)$. Therefore we obtained the first possibility.

Let now $(g,h)$ and $(g',h')$ be nonadjacent. If they are not in the same ($G$- or $H$-) layer, then $(g,h')$ and $(g',h)$ are not adjacent to $(g,h)$ nor to $(g',h')$ and must have label 1. All the other vertices must be in $V_0$. But, this is a contradiction because $V_1\cup V_2$ induces four isolated vertices. Hence, $(g,h)$ and $(g',h')$ are in the same $G$- or $H$-layer, say in $H^g$. So, $g=g'$. If there exists different $h_1,h_2\in V(H)-\{h,h'\}$, then $(g,h_1), (g,h_2)\in V_1$, since there are no edges between vertices of $H^g$. A contradiction again, due to no existing edges between vertices of $V_1\cup V_2$. If $V(H)|=3$, say $V(H)=\{h,h',h_1\}$, then $f(g,h_1)=1$ and the other vertex $(a,b)$ from $V_1$ must be adjacent to all three vertices from $H^g$. This is not possible as $(a,b)$ is contained in one of the layers $G^h$, $G^{h'}$ or $G^{h_1}$. Again we have a vertex from $V_1\cup V_2$ that is not adjacent to any other vertex of $V_1\cup V_2$, a contradiction. So, $H$ contains only two vertices $h$ and $h'$, which are adjacent and therefore both universal vertices. If both vertices from $V_1$ belong to the same $G$-layer, say $G^h$, then $(g,h)$ is not adjacent to any vertex from $V_1\cup V_2$, which is not possible. So, we may assume that $V_1=\{(g_1,h),(g_2,h')\}$. Clearly $gg_1,gg_2\in E(G)$, so that $V_1\cup V_2$ induces a subgraph without isolated vertices. Also every vertex $(g_3,h)\in V_0$ must be adjacent to $(g,h')$, which means that $gg_3\in E(G)$ and $g$ is an universal vertex of $G$. (Notice also that in the case when $g_1=g_2$, there always exists $g_3\in V(G)-\{g,g_1\}$, because otherwise we have a contradiction with $(ii)$.) This yields the middle case of $(iii)$.

To end with $(iii)$ let $|V_1|=4$ and $|V_2|=1$, where $V_2=\{(g,h)\}$. Let $(g',h')\in V_1$ be a neighbor of $(g,h)$. Clearly all vertices from $G^h\cup H^g$ must be in $V_1\cup V_2$, meaning that one of the factors is $K_2$ and the other contains three vertices, say $H\cong K_2$. Moreover, $g$ must be a universal vertex of $G$. So, either $G\cong C_3$ or $G\cong P_3$, which is the middle case of $(iii)$ and the proof of $(iii)$ is completed.

We continue with $(\Leftarrow)$ of $(iv)$. We may assume that $G$ has exactly one universal vertex $g$, and that $H$ is different from $K_2$ with a universal vertex $h$, and that at most one of $G$ and $H$ is triangle centered. Further, let $g'$ and $h'$ be arbitrary neighbors of $g$ in $G$ and of $h$ in $H$, respectively. By $(i),(ii)$ and $(iii)$ we know that $\gamma_{tR}(G\times H)\geq 7$. If we set $V_2=\{(g,h),(g,h'),(g',h)\}$, $V_1=\{(g',h')\}$ and $V_0=V(G\times H)-(V_1\cup V_2)$, then $f=(V_0,V_1,V_2)$ is a total Roman dominating function with $\omega(f)=7$. Hence, $\gamma_{tR}(G\times H)\leq 7$ and the equality follows.

For $(\Rightarrow)$ of $(iv)$, suppose that $\gamma_{tR}(G\times H)=7$ and that $f=(V_0,V_1,V_2)$ is a $\gamma_{tR}(G\times H)$-function. First assume that $|V_1|=1$ and $|V_2|=3$, where $V_1=\{(g_1,h_1)\}$ and $V_2=\{(g_2,h_2),(g_3,h_3),(g_4,h_4)\}$. We may also assume that $(g_1,h_1)(g_2,h_2),(g_3,h_3)(g_4,h_4)\in E(G\times H)$ as $f$ is a $\gamma_{tR}(G\times H)$-function. Vertices $(g_3,h_4)$ and $(g_4,h_3)$ are not adjacent to $(g_3,h_3)$ nor to $(g_4,h_4)$. If $g_3\neq g_2\neq g_4$, then $(g_2,h_2)$ is adjacent to both $(g_3,h_4)$ and $(g_4,h_3)$ (even if one of them equals to $(g_1,h_1)$). As a consequence, we have $g_2g_3,g_2g_4\in E(G)$ and $h_2h_3,h_2h_4\in E(H)$. In other words, $g_2g_3g_4$ and $h_2h_3h_4$ form a triangle in $G$ and $H$, respectively. Let $g$ be an arbitrary vertex from $V(G)-\{g_2,g_3,g_4\}$ and let $h$ be an arbitrary vertex from $V(H)-\{h_2,h_3,h_4\}$. The vertex $(g,h)$ is adjacent to at least one vertex from $V_2$ (even if $(g,h)=(g_1,h_1)$). Let $(g_i,h_i)$ be a neighbor of $(g,h)$ for some $i\in \{2,3,4\}$. Clearly, $(g_i,h)$ and $(g,h_i)$ are not adjacent to $(g_i,h_i)$. Hence they must be adjacent to $(g_j,h_j)$ for some $j\in\{2,3,4\}-\{i\}$, meaning that $gg_j\in E(G)$ and $hh_j\in E(H)$. We see that both $G$ and $H$ are triangle centered graphs, and by $(iii)$ we have $\gamma_{tR}(G\times H)=6$, a contradiction with $\gamma_{tR}(G\times H)=7$.

So we can assume that either $g_2=g_3$ or $g_2=g_4$, say that $g_2=g_3$. Moreover, also $h_2=h_4$ as otherwise $(g_2,h_4)$ has no neighbor of label $2$. If $h_2$ is not adjacent to some vertex $h\in V(H)$, then $(g_2,h)$ is not adjacent to a vertex of label $2$, meaning that $h_2$ is a universal vertex of $H$. Similarly, we see that $g_2$ is a universal vertex of $G$. We have $\gamma_{tR}(G\times H)=6$ by $(iii)$ when both $G$ and $H$ have (at least) two universal vertices, or one is $K_2$ and the other contains a universal vertex, a contradiction. Hence, one of $G$ or $H$ has at most one universal vertex and the other is not $K_2$ and we are done in this case.

The second possibility is that $|V_1|=3$ and $|V_2|=2$, where $V_1=\{(g_1,h_1),(g_2,h_2),(g_3,h_3)\}$ and $V_2=\{(g_4,h_4),(g_5,h_5)\}$. If $(g_4,h_4)$ and $(g_5,h_5)$ are adjacent, then $(g_4,h_5),(g_5,h_4)\in V_1$, say $(g_4,h_5)=(g_2,h_2)$ and $(g_5,h_4)=(g_3,h_3)$. Suppose that $g_1\notin \{g_4,g_5\}$ and $h_1\notin \{h_4,h_5\}$. All the vertices of $G^{h_4}-\{(g_4,h_4),(g_5,h_4)\}$ must be in $V_0$ and adjacent to $(g_5,h_5)$, meaning that $g_5$ is a universal vertex of $G$. Similarly, all the vertices of $G^{h_5}-\{(g_4,h_5),(g_5,h_5)\}$ must be in $V_0$ and adjacent to $(g_4,h_4)$, meaning that $g_4$ is a universal vertex of $G$. This means that $G$ is triangle centered with central triangle $g_1g_4g_5$. By symmetric arguments $H$ is triangle centered with central triangle $h_1h_4h_5$. By $(iii)$ we have $\gamma_{tR}(G\times H)=6$, a contradiction. So, either $h_1\in\{h_4,h_5\}$ or $g_1\in \{g_4,g_5\}$, say $h_1=h_4$. By the same arguments as above, we see that $g_4$ is a universal vertex of $G$, and that $h_4$ and $h_5$ are universal vertices of $H$. (Notice that $g_1$ is not adjacent to $g_5$, otherwise also $g_5$ is universal vertex, a contradiction with $(iii)$.) If $H\cong K_2$, then we have $\gamma_{tR}(G\times H)=6$ by $(iii)$, a contradiction. Otherwise $H\ncong K_2$ and we are done.

Now we can assume that $(g_4,h_4)$ and $(g_5,h_5)$ are not adjacent. If $g_4\neq g_5$ and $h_4\neq h_5$, then, as in the previous paragraph, we can choose the notation such that $(g_4,h_5)=(g_2,h_2)$ and that $(g_5,h_4)=(g_3,h_3)$. Moreover, $(g_1,h_1)$ must be adjacent to all other vertices from $V_1\cup V_2$ in order to avoid isolated vertices of positive label. Vertices $(g_5,h_1)$ and $(g_1,h_4)$ are from $V_0$ and must have a neighbor in $V_2$. The only possibility is that $(g_5,h_1)$ is adjacent to $(g_4,h_4)$ and $(g_1,h_4)$ is adjacent to $(g_5,h_5)$. The mentioned edges imply that $g_4g_5\in E(G)$ and $h_4h_5\in E(H)$, a contradiction with the not adjacency of $(g_4,h_4)$ and $(g_5,h_5)$. It remains that $(g_4,h_4)$ and $(g_5,h_5)$ belong to the same layer, say $H^{g_4}$, that is $g_4=g_5$. Every vertex from $H^{g_4}-\{(g_4,h_4),(g_4,h_5)\}$ is not adjacent to a vertex of label $2$ and must poses label $1$. We need also at least two vertices of label $1$ outside of $H^{g_4}$ to assure non isolated vertices in $V_1\cup V_2$. This means that $|V(H)|\leq 3$. Every vertex from $G^{h_4}-\{(g_4,h_4)\}$ must be adjacent to $(g_4,h_5)$ and $g_4$ is a universal vertex of $G$. If $H\cong K_2$, then we have a contradiction with $(iii)$. So either $H\cong P_3$ or $H\cong C_3$, meaning that also $H$ has a universal vertex and the second possibility is done.

The last option is that $|V_1|=5$ and $|V_2|=1$, where $V_2=\{(g,h)\}$. Clearly all vertices from $G^h\cup H^g$ must be in $V_1\cup V_2$ and $g$ and $h$ must be universal vertices of $G$ and $H$, respectively. We either obtain a contradiction with $(iii)$ (when one factor is $K_2$) or obtain that $G\cong H\cong K_{1,2}$ which yields the desired situation and the proof of $(iv)$ is completed.

We conclude this proof with $(v)$. We have $\gamma_{tR}(G\times H)\geq 8$ from assertions $(i)-(iv)$. Let $D_G=\{g,g'\}$ be a $\gamma_t(G)$-set and $D_H=\{h,h'\}$ be a $\gamma_t(H)$-set. We set $V_2=D_G\times D_H$, $V_1=\emptyset$ and $V_0=V(G\times H)-V_2$ and claim that $f=(V_0,V_1,V_2)$ is a total Roman dominating function on $G\times H$. Let $(g_1,h_1)\in V_0$. Clearly, $g_1$ is neighbor of $g$ or of $g'$, say of $g$, and $h_1$ is neighbor of $h$ or $h'$, say $h$. Therefore $(g_1,h_1)$ is neighbor of $(g,h)$ and $f$ satisfies the conditions to be total Roman dominating for $G\times H$. Hence, the inequality $\gamma_{tR}(G\times H)\leq 8$ is obtained, which leads to the claimed equality.
\end{proof}
%The final direction is $(\Rightarrow)$. For this let $\gamma_{tR}(G\times H)=8$ and let $f=(V_0,V_1,V_2)$ be a $\gamma_{tR}(G\times H)$-function. If both graphs $G$ and $H$ have a universal vertex, then we have a contradiction with $(ii)$, $(iii)$ or $(iv)$.  Assume first that $|V_1|=0$ and $|V_2|=4$ where $V_1\cup V_2=V_2=\{(g_1,h_1),(g_2,h_2),(g_3,h_3),(g_4,h_4)\}$. We may also assume that $(g_1,h_1)(g_2,h_2),(g_3,h_3)(g_4,h_4)\in E(G\times H)$. If $(g_3,h_3)=(g_1,h_2)$ and $(g_4,h_4)=(g_2,h_1)$, then every vertex from

A \emph{wheel} graph $W_n$, $n\geq 4$, is a join of $K_1$ and $C_{n-1}$ and a \emph{fan} graph $F_n$, $n\geq 2$, is a join of $K_1$ and $P_{n-1}$. Clearly $W_n$ and $F_n$ have exactly one universal vertex when $n>4$. In particular, $W_n$ and $F_n$ are triangle centered whenever $n\in \{4,5\}$. For a complete graph $K_n$ and a maximum matching $M$ of it, the graph $K_n-M$, $n\geq 5$, is a triangle centered graph with a universal vertex whenever $n$ is an odd number. By using Theorem \ref{small} we directly obtain the next results (among others).

\begin{corollary}\label{exact} For integers $n,m>5$, $p\geq 1$, $q,s,t\geq 2$, $r>2$ and maximum matchings $M$ and $M'$ we have
\begin{itemize}
\item[(i)] $\gamma_{tR}(K_r\times K_s)=6$;
\item[(ii)] $\gamma_{tR}(K_{1,s}\times K_{1,t})=7$;
\item[(iii)] $\gamma_{tR}(K_{p,q}\times K_{s,t})=8$;
\item[(iv)] $\gamma_{tR}(K_q\times K_{s,t})=8$;
\item[(v)] $\gamma_{tR}(K_r\times W_n)=7$;
\item[(vi)] $\gamma_{tR}(K_r\times F_n)=7$;
\item[(vii)] $\gamma_{tR}(W_n\times F_m)=8$;
\item[(viii)] $\gamma_{tR}(W_n\times W_m)=8$;
\item[(ix)] $\gamma_{tR}(F_n\times F_m)=8$;
\item[(x)] $\gamma_{tR}((K_n-M)\times (K_m-M'))=6$.
\end{itemize}
\end{corollary}

With the help from Corollary \ref{exact}, we can comment the sharpness for most of the bounds from Section \ref{section-general}. The upper bounds of Theorem \ref{th:general bounds}, of Corollary \ref{cor1} and of Remark \ref{remark} are sharp by $(ii)$ of Corollary \ref{exact}. The upper bound from Corollary \ref{total} is sharp by $(iii),(iv),(vii),(viii)$ and $(ix)$ of Corollary \ref{exact}. For $p=q=s=t=2$ we have $\gamma_{tR}(K_{2,2}\times K_{2,2})=\gamma_{tR}(C_4\times C_4)=8$ by $(iii)$ of Corollary \ref{exact}, and so for Corollary \ref{cor2}, its upper bound is sharp. The lower bound from Theorem \ref{th:general bounds} follows from $\gamma_{tR}(P_4\times P_4)=8=\rho(P_4)\gamma_{tR}(P_4)$ which holds by $(v)$ of Theorem \ref{small}. By $(iii)$ of Corollary \ref{exact}, we show the sharpness of the bounds from Theorems \ref{th:triangle-free-bipar} and \ref{th:triangle-free-regular} and Corollary \ref{th:two-EOD}. In conclusion, only the tightness of the bound presented in Theorem \ref{EOD} remains open.

We end this section with an alternative presentation with respect to Theorem \ref{small}, where we consider the number of vertices in $V_1\cup V_2$ of a total Roman dominating function. For the minimum cardinality of $V_1\cup V_2$, we need an additional condition that the cardinality of $V_2$ must be maximum to be able to characterize them.

\begin{theorem}\label{triangel}
Let $f=(V_0,V_1,V_2)$ be a $\gamma_{tR}(G\times H)$-function with the largest possible cardinality for $V_2$, where $G$ and $H$ are two graphs of order at least three. The next items are equivalent.
\begin{itemize}
\item[(i)] Graphs $G$ and $H$ are triangle centered.
\item[(ii)] $\gamma_{tR}(G\times H)=6$.
\item[(iii)] $|V_1\cup V_2|=3$.
\end{itemize}
\end{theorem}

\begin{proof}
The direction $((i)\Rightarrow (ii))$ follows from $(iii)$ of Theorem \ref{small}.

For the direction $((ii)\Rightarrow (iii))$, let $\gamma_{tR}(G\times H)=6$ where $f=(V_0,V_1,V_2)$ is a $\gamma_{tR}(G\times H)$-function with maximum cardinality of $V_2$. There exist vertices from $G\times H$ in $V_0$ as there are at least nine vertices in $G\times H$. Consequently $V_2\neq\emptyset$. Let $(g,h)\in V_2$ and let $(g',h')$ be a neighbor of $(g,h)$ with $f(g',h')>0$. There exists at least one vertex $(x,y)$ from $(G^h\cup H^g)-\{(g,h)\}$ of label 0, because $\gamma_{tR}(G\times H)=6$. Suppose that $(g'',h'')$ is a neighbor of $(x,y)$ of label 2. Assume first that $(g',h')=(g'',h'')$. The vertices $(g',h)$ and $(g,h')$ are not adjacent to $(g',h')$ nor to $(g,h)$. If they have label equal to $1$, then all the other vertices have label $0$ and every vertex is adjacent to $(g,h)$ or to $(g',h')$. Let $g_1$ and $h_1$ be a third vertex of $G$ and $H$, respectively. Clearly, $(g_1,h')$ and $(g',h_1)$ are adjacent to $(g,h)$ and with this, we have $gg_1\in E(G)$ and $hh_1\in E(H)$. Similarly, $(g,h_1)$ and $(g_1,h)$ are adjacent to $(g',h')$, and with this we get $g'g_1\in E(G)$ and $h'h_1\in E(H)$. Let us define $f'=(V'_0,V'_1,V'_2)$ where $V'_0=(V_0\cup V_1)-\{(g_1,h_1)\}$, $V'_1=\emptyset$ and $V'_2=V_2\cup \{(g_1,h_1)\}$. Clearly, $f'$ is a total Roman dominating function with $|V'_2|>|V_2|$, a contradiction with the choice of $f$. Therefore, the label of $(g',h)$ and $(g,h')$ must be $0$ and there exists a third vertex $(g_2,h_2)$ of label $2$ that is adjacent to $(g',h)$ and $(g,h')$. From $\gamma_{tR}(G\times H)=6$ it follows that $|V_1\cup V_2|=3$.

Next we assume that $(g',h')\neq(g'',h'')$. If also $f(g',h')=2$, then $V_2=\{(g,h),(g',h'),(g'',h'')\}$ and $V_1=\emptyset$ and we are done. So let $f(g',h')=1$. Because $\gamma_{tR}(G\times H)=6$ there exists a fourth vertex $(a,b)$ in $V_1\cup V_2$ with $f(a,b)=1$ and all other vertices are in $V_0$. Vertex $(g'',h'')$ is not from $G^h\cup H^g$, because $V_2$ contains only $(g,h)$ and $(g'',h'')$ and we have at least three vertices in every $G$- or $H$-layer. Hence, $g\neq g''$ and $h\neq h''$. Vertices $(g,h'')$ and $(g'',h)$ are not adjacent to $(g,h)$ nor to $(g'',h'')$, and must therefore have label $1$. This leads to $\{(g'',h),(g,h'')\}=\{(g',h'),(a,b)\}$, and this is not possible since $(g',h')$ is adjacent to $(g,h)$. Hence, $|V_1\cup V_2|=3$ in all cases and this implication is done.

$((iii)\Rightarrow (i))$ Let $|V_1\cup V_2|=3$ and let $(g_1,h_1),(g_2,h_2),(g_3,h_3)\in V_1\cup V_2$. As $V_1\cup V_2$ induces a graph without isolated vertices, one vertex of these mentioned three, say $(g_2,h_2)$, must be adjacent to the other two. Thus, $g_1g_2,g_2,g_1\in E(G)$ and $h_1h_2,h_2,h_3\in E(H)$. If $g_1g_3\notin E(G)$, then $(g_1,h_2)$ is a vertex that is labeled with $0$ being not neighbor of a vertex belonging to $V_2$. Similarly, if $h_1h_3\notin E(H)$, then $(g_2,h_1)$ is a vertex whose label is equal to $0$ being not neighbor of one vertex from $V_2$. Hence $g_1g_2g_3$ and $h_1h_2h_3$ form a triangle in $G$ and $H$, respectively. Suppose there is a vertex $g\in V(G)$ which is either neighbor of exactly one vertex of $\{g_1,g_2,g_3\}$, say to $g_1$, or to no vertex of $\{g_1,g_2,g_3\}$. In both cases the vertex $(g,h_1)$ has label $0$ and is not adjacent to any vertex of $V_1\cup V_2$, which is not possible since $f$ is a function which is total Roman dominating. Hence, every vertex $g\in V(G)$ is adjacent to two or more vertices from $\{g_1,g_2,g_3\}$ and $G$ is triangle centered. Similarly, one shows that $H$ is triangle centered.
\end{proof}

%%%%%%%%%%%%%%%%%%%%%%%%%%%%%%%%%%%%%%%%%%%%%%%%%%%%%%%%%%%%%%5

\section{A general lower bound and its consequences on the direct product}

The following lower bound for $\gamma_{tR}(G)$ depends on the order of $G$ and its maximum degree $\Delta(G)$ as well as on a $\gamma_{tR}(G)$-function.

\begin{theorem}\label{genlower}
If $f=(V_0,V_1,V_2)$ is a $\gamma_{tR}(G)$-function of a graph $G$, then $\gamma_{tR}(G)\geq |V(G)|-(\Delta(G)-2)|V_2|$ and $|V_2|\geq\frac{|V(G)|-|V_1|}{\Delta(G)}$. Moreover, if in addition $|V(G)|=\Delta(G)|V_2|+|V_1|$, then the equality $\gamma_{tR}(G)=|V(G)|-(\Delta(G)-2)|V_2|$ holds.
\end{theorem}

\begin{proof}
Assume $g=(V_0,V_1,V_2)$ is a $\gamma_{tR}(G)$-function. Every vertex from $V_2$ must have one neighbor in $V_1\cup V_2$. This means that every  vertex from $V_2$ has no more than $\Delta(G)-1$ adjacent vertices in $V_0$. With this we have
\begin{equation}
|V(G)|=|V_0|+|V_1|+|V_2|\leq (\Delta(G)-1)|V_2|+|V_1|+|V_2|. \label{one}
\end{equation}
From (\ref{one}) we extract $|V_2|$ and obtain the second inequality
\begin{equation*}
|V_2|\geq\frac{|V(G)|-|V_1|}{\Delta(G)}.
\end{equation*}
Notice that from (\ref{one}), it follows $|V_2|$ is maximum when $|V_1|=0$. Now we return to (\ref{one}), and add $0=|V_2|-|V_2|$ on the right side to get
\begin{equation}
|V(G)|\leq (\Delta(G)-2)|V_2|+2|V_2|+|V_1|=|V_2|(\Delta(G)-2)+\gamma_{tR}(G), \label{two}
\end{equation}
that yields the first inequality. Notice that from the additional condition $|V(G)|=\Delta(G)|V_2|+|V_1|$ we get
\begin{equation*}
|V_0|+|V_1|+|V_2|=|V(G)|=\Delta(G)|V_2|+|V_1|
\end{equation*}
and consequently $|V_0|=(\Delta(G)-1)|V_2|$. This connection gives the equality in the lines (\ref{one}) and (\ref{two}) and the proof is completed.
\end{proof}

If we rewrite the Theorem \ref{genlower} for the direct product $G\times H$, then we have the following.

\begin{corollary}\label{lower1}
Let $G$ and $H$ be any two graphs. If $g=(V'_0,V'_1,V'_2)$ is a $\gamma_{tR}(G\times H)$-function, then $\gamma_{tR}(G\times H)\geq |V'(G)||V'(H)|-(\Delta(H)\Delta(G)-2)|V'_2|$ and $|V'_2|\geq\frac{|V'(G)||V'(H)|-|V'_1|}{\Delta(H)\Delta(G)}$. Moreover, if in addition $|V'(G)||V'(H)|=\Delta(H)\Delta(G)|V'_2|+|V'_1|$, then the equality $\gamma_{tR}(G\times H)=|V'(G)||V'(H)|-(\Delta(H)\Delta(G)-2)|V'_2|$ holds.
\end{corollary}

The lower bound from Theorem \ref{genlower} is better when $|V_2|$ is small as possible. Also, one cannot expect that the mentioned bound behave well when there exists a small quantity of vertices with maximum number of neighbors in $G$. From this point of view, one can expect that Theorem \ref{genlower} works at its best for regular graphs. To see this, the following known remark is necessary.

\begin{remark}\label{remark-efficient-open-graph}{\rm \cite{KuPeYe1}}
If $S$ is an efficient open dominating set of an efficient open domination graph $G$, then $S$ is a $\gamma_t(G)$-set.
\end{remark}

\begin{theorem}\label{regular}
If $G$ is a regular efficient open domination graph, then $\gamma_{tR}(G)=2\gamma_t(G)$.
\end{theorem}

\begin{proof}
Let $D$ be an efficient open dominating set of an $r$-regular graph $G$. By Remark \ref{remark-efficient-open-graph} we have that $D$ is a $\gamma_t(G)$-set. Hence, $f=(V_0,V_1,V_2)=(V(G)-D,\emptyset,D)$ is a total Roman dominating function on $G$ of weight $\omega(f)=2\gamma_t(G)$ that clearly fulfills the condition $|V(G)|=\Delta(G)|V_2|+|V_1|=r|D|$. By Theorem \ref{genlower} the result follows.
\end{proof}

For two graphs $G$ and $H$, its direct product $G\times H$ represents an efficient open domination graph whenever both $G$ and $H$ contains efficient open dominating sets. This was proved in \cite{AbHamTay}. Moreover, for the two efficient open dominating sets $D_G$ and $D_H$ of $G$ and $H$, respectively, the set $D_G\times D_H$ is an efficient open dominating set of $G\times H$. Hence we have the following result.

\begin{corollary}\label{EOD2}
If $G$ and $H$ are regular graphs and they are also efficient open domination graphs, then $\gamma_{tR}(G\times H)=2\gamma_t(H)\gamma_t(G)$.
\end{corollary}

The relaxation of Corollary \ref{EOD2} and Theorem \ref{regular} without the condition of regular graphs is not true anymore as shown by $(ii)$ of Corollary \ref{exact}. Clearly $K_{1,s}$ and $K_{1,t}$ are efficient open domination graphs that are not regular and we have $\gamma_{tR}(K_{1,s}\times K_{1,t})=7\neq 8=2\gamma_t(K_{1,s})\gamma_t(K_{1,t})$.

A \emph{prism} $P_G$ over a graph $G$ is a graph obtained from two disjoint copies of the graph $G$ by adding a perfect matching between analogous vertices of each copy (or the Cartesian product $G\Box K_2$). All the prisms that are efficient open domination graphs are described in Theorem 4.3 from \cite{KuPeYe1}. One $3$-regular example is $P_{C_{3r}}$ and for them we have $\gamma_t(P_{C_{3r}})=2r$.

It is well known that a cycle $C_n$ contains an efficient open dominating set whenever $n$ is congruent with $0$ modulo $4$. Thus, the next result is clear by Corollary \ref{EOD2}.

\begin{corollary}\label{EOD3}
If $m$ and $n$ are positive integers divisible by $4$ and $t\geq 2$ and $r\geq 1$ are any integers, then
\begin{itemize}
\item[(i)] $\gamma_{tR}(C_m\times C_n)=\frac{mn}{2}$;
\item[(ii)] $\gamma_{tR}(C_m\times K_{t,t})=2m$;
\item[(iii)] $\gamma_{tR}(C_m\times P_{C_{3r}})=2mr$;
\item[(iv)] $\gamma_{tR}(K_{t,t}\times P_{C_{3r}})=8r$.
\end{itemize}
\end{corollary}

\end{document}